\documentclass[11pt]{article}
\usepackage{amsmath, amsfonts, amssymb, amsthm, mathtools, tikz, bm} %The usual suspects (Hudson). 
\usepackage[numbers,square]{natbib} %For citations having "post" and "pre" information. (Hudson)
\usepackage{hyperref, fullpage}
\usepackage[size=footnotesize]{caption} %To make captions centered using \captionsetup{} (Hudson)
\usepackage{verbatim}

\newcommand{\card}[1]{\lvert #1\rvert} %Cardinality Shortcut Command (Hudson)
\newtheorem{theorem}{Theorem}[section]
\newtheorem{lemma}[theorem]{Lemma}
\newtheorem*{mainthm}{Main Theorem}
\newtheorem*{brooks}{Brooks' Theorem}
\newtheorem*{BK}{Borodin--Kostochka Conjecture}
\newtheorem*{copycatlem}{Copycat Lemma}
\newcommand\vph{\varphi}
\newcommand\G{\mathcal{G}}
\newcommand\HH{\mathcal{H}}

\tikzstyle{uStyle}=[fill=white, circle, thick, draw=black, inner sep=1.2pt]
\tikzstyle{lStyle}=[fill=none, draw=none, circle, inner sep=15pt]

\begin{document}
\author{Daniel W. Cranston\thanks{%
Department of Mathematics and Applied
Mathematics, Viriginia Commonwealth University, Richmond, VA;
\texttt{dcranston@vcu.edu}} 
\and
Hudson Lafayette\thanks{
Department of Mathematics and Applied
Mathematics, Viriginia Commonwealth University, Richmond, VA;
\texttt{lafayettehl@vcu.edu}} 
\and
Landon Rabern\thanks{
%Franklin \& Marshall College, Lancaster, PA;
%LBD Data Solutions, Lancaster, PA;
\texttt{landon.rabern@gmail.com}}
}
\title{Coloring $(P_5, \text{gem})$-free graphs with $\Delta -1$ colors}
\maketitle
\begin{abstract}
The Borodin--Kostochka Conjecture states that for a graph $G$, if
$\Delta(G)\geq 9$ and $\omega(G)\leq \Delta(G)-1$, then
$\chi(G)\leq\Delta(G) -1$. We prove the Borodin--Kostochka Conjecture
for $(P_5, \text{gem})$-free graphs, i.e., graphs with no induced $P_5$ and
no induced $K_1\vee P_4$.
\end{abstract}

\section{Introduction}

\indent Every graph $G$ satisifies $\chi(G)\le \Delta(G) +1$. (We denote
by $\Delta(G)$, $\omega(G)$, and $\chi(G)$
the \textit{maximum degree}, \textit{clique number}, and  \textit{chromatic
number} of $G$.)
To see this, we greedily coloring the vertices of $G$ in any order.
In 1941, Brooks~\cite{Brooks}  strengthened this bound to $\chi(G)\le \Delta(G)$.

\begin{brooks}
Let $G$ be a graph. If $\Delta(G)\geq 3$ and $\omega(G)\leq \Delta(G)$, then $\chi(G) \leq \Delta(G)$. 
\end{brooks}

 In 1977, Borodin and Kostochka~\cite{BK77} conjectured a further strengthening of Brooks' bound.

\begin{BK}
Let $G$ be a graph. If $\Delta(G)\geq 9$ and $\omega(G)\leq \Delta(G)-1$, then $\chi(G) \leq \Delta(G)-1$.
\end{BK}

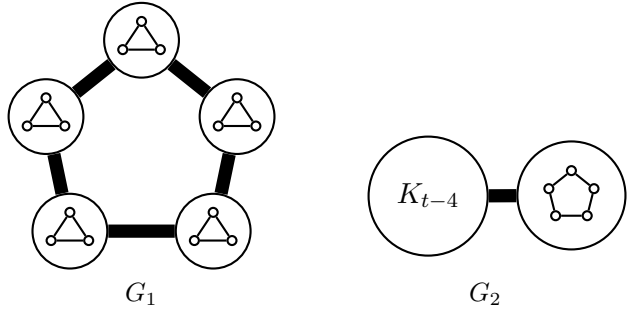
\begin{figure}[t]
\centering
\captionsetup{margin=.8in} 
  \begin{tikzpicture}[scale=.5]
 \tikzset{every node/.style=uStyle}
   \node (a) at (0,.75) {}; 
   \node (b) at (.5,0) {};
   \node (c) at (-.5,0) {}; \draw[thick, black] (a)--(b)--(c)--(a);
   \node[fill=none, draw=black, thick, circle, inner sep=10pt] (c1) at (0,.25) {};

   \begin{scope}[xshift=1in,yshift=-.8in]
   \node (a) at (0,.75) {}; 
   \node (b) at (.5,0) {};
   \node (c) at (-.5,0) {}; \draw[thick, black] (a)--(b)--(c)--(a);
   \node[fill=none, draw=black, thick, circle, inner sep=10pt] (c2) at (0,.25) {};
   \end{scope}

   \begin{scope}[xshift=.75in,yshift=-2in]
   \node (a) at (0,.75) {}; 
   \node (b) at (.5,0) {};
   \node (c) at (-.5,0) {}; \draw[thick, black] (a)--(b)--(c)--(a);
   \node[fill=none, draw=black, thick, circle, inner sep=10pt] (c3) at (0,.25) {};
   \end{scope}

   \begin{scope}[xshift=-.75in,yshift=-2in]
   \node (a) at (0,.75) {}; 
   \node (b) at (.5,0) {};
   \node (c) at (-.5,0) {}; \draw[thick, black] (a)--(b)--(c)--(a);
   \node[fill=none, draw=black, thick, circle, inner sep=10pt] (c4) at (0,.25) {};
   \end{scope}

   \begin{scope}[xshift=-1in,yshift=-.8in]
   \node (a) at (0,.75) {}; 
   \node (b) at (.5,0) {};
   \node (c) at (-.5,0) {}; \draw[thick, black] (a)--(b)--(c)--(a);
   \node[fill=none, draw=black, thick, circle, inner sep=10pt] (c5) at (0,.25) {};
   \end{scope}
   \draw[line width=5pt, black] (c1)--(c2)--(c3)--(c4)--(c5)--(c1);
   \node[lStyle,inner sep=0pt] (g1) at (0,-6.5) {\small$G_1$};
   
   \begin{scope}[xshift=3in, scale=1.2, yshift=-1.2in]
   \node[circle, thick, draw=black, fill=white, inner sep=.1in] (1) at (0,-.2) {$K_{t-4}$}; 
   \begin{scope}[xshift=1.25in, yshift=-.25in]
    \node[circle, thick, draw=black, fill=white, inner sep=.2in] (2) at (0,.45) {}; 
    \node (a) at (0,1) {}; 
    \node (b) at (.5,.6) {};
    \node (c) at (.35,0) {}; 
    \node (d) at (-.35,.0) {};
    \node (e) at (-.5,.6) {}; 
    \draw[thick, black] (a)--(b)--(c)--(d)--(e)--(a);
   \end{scope}
   \draw[line width=5pt, black] (1)--(2); 
   \node[lStyle,inner sep=0pt] at ([shift=({0:3in})]g1) {\small$G_2$};
   \end{scope}
  \end{tikzpicture}
  \caption{Each bold edge denotes a complete bipartite graph.  
We have $\Delta(G_1)=8$, $\omega(G_1)=6$, but $\chi(G_1)=8$.  And
we have $\Delta(G_2)=t$, $\omega(G_2)=t-2$, but $\chi(G_2)=t-1$.}
 \label{fig:fig1}
\end{figure}

If true, the Borodin--Kostochka Conjecture is best possible in the following
two ways.  First, if the hypothesis $\Delta(G)\geq 9$ is weakened to
$\Delta(G)\geq 8$ then the conjecture is false, as witnessed by 
$G_1$, on the left in Figure \ref{fig:fig1}. Note that $G_1$ is a counterexample
to this stronger version since $\Delta(G_1)=8$ and
$\omega(G_1)=6\leq 7=\Delta(G_1)-1$, but $\chi(G)=8>7=\Delta(G_1)-1$. 
Second, if the
hypothesis $\omega(G)\leq \Delta(G) -1$ is strengthened to $\omega(G)\leq
\Delta(G)-2$, then the bound $\chi(G)\leq \Delta(G)-1$ cannot be
strengthened to $\chi(G)\leq \Delta(G)-2$, as witnessed by $G_2$, on the
right in Figure~\ref{fig:fig1}. 
For each $t\geq 9$, note that $G_2$ is a counterexample to this stronger
version, since $\Delta(G)=t\geq9$ and $\omega(G_2)= \Delta(G_2)-2$, but
$\chi(G_2)=\Delta(G_2)-1$.

By Brooks' Theorem, each graph $G$ with $\chi(G)>\Delta(G) \geq 9$
contains $K_{\Delta(G)+1}$.
So the Borodin--Kostochka Conjecture asserts that each $G$ with
$\chi(G)=\Delta(G)\geq 9$ contains $K_{\Delta(G)}$. 
This conjecture has been proved for many interesting classes of graphs,
particularly those defined by forbidden subgraphs.
In 2013, Cranston and Rabern~\cite{CR13clawfree} proved it for claw-free graphs.  

\begin{theorem}[\cite{CR13clawfree}]
Every claw-free graph with $\chi(G) \geq \Delta(G) \geq 9$ contains $K_{\Delta(G)}$.
\end{theorem}

The strongest partial result toward the Borodin--Kostochka Conjecture is due to
Reed~\cite{ReedBK}.  In 1999, he proved this conjecture for every graph $G$ with $\Delta(G)$
sufficiently large.

\begin{theorem}[\cite{ReedBK}] 
Every graph with $\chi(G)=\Delta(G)\geq 10^{14}$ contains $K_{\Delta(G)}$.
\end{theorem}

%\noindent 
Reed stated that a more careful analysis of his argument could reduce the lower
bound to about $10^3$, but definitely not to $10^2$. 

In 1998, Reed also
conjectured  \cite{Reed} that $\chi(G)\leq \lceil (\omega(G)+\Delta(G)+1)/2\rceil$. 
This conjecture, now called Reed's Conjecture, is weaker than the Borodin--Kostochka Conjecture
when $\omega(G)\in \{\Delta(G)-1, \Delta(G)-2\}$, equivalent to it when 
$\omega(G)\in \{\Delta(G)-3, \Delta(G)-4\}$, and stronger otherwise.
Chudnovsky, Karthick, Maceli, and Maffray~\cite{CKMM19} proved Reed's Conjecture for all
$(P_5,\textrm{gem})$-free graphs.  In this note we use their structure theorem for
$(P_5,\text{gem})$-free graphs, as well as one of their key lemmas
\cite[][Theorem 3 and Lemma 2]{CKMM19}, to prove the Borodin--Kostochka
Conjecture for the same class of graphs.

\begin{mainthm} 
Let $G$ be a $(P_5,\text{gem})$-free graph. If $\Delta(G)\geq 9$ and
$\omega(G)\leq \Delta -1$, then $\chi(G)\leq \Delta(G)-1$.
\end{mainthm}

%%%%%%%%%%%%%%%%%%%%%%%%%%%%%%%%%%%%%%%%%%%%%%%%%%%%
%%%%%%%%%%%%                        Definitions and Notation                           %%%%%%%%%%%%
%%%%%%%%%%%%%%%%%%%%%%%%%%%%%%%%%%%%%%%%%%%%%%%%%%%%
\section{Definitions and Notation}
	
All of our graphs are finite and simple. 
Most of our definitions follow \cite{West}, but we highlight a few terms.
By \emph{coloring} we mean a
proper vertex coloring. A graph is \emph{$k$-colorable} (or has a
\emph{$k$-coloring}) if there exists a coloring $\varphi : V(G) \rightarrow
\{1,\dots, k\}$. 
Given two graphs $S$ and $T$, the \emph{join}, denoted $S\vee
T$ is formed from their disjoint union by adding all edges with one endpoint in
$S$ and the other in $T$.
For graphs $H_1, \dots, H_s$, a graph $G$ is $(H_1, \dots, H_s)$-free if
$G$ does not contain any of $H_1, \dots, H_s$ as an induced subgraph. Thus,
the class of $(P_5, \text{gem})$-free graphs is all graphs that do not
contain an induced $P_5$ or an induced $\text{gem}$. 
Here \textit{gem} is $K_1\vee P_4$; see Figure 2.

\begin{figure}[!h]
\centering
  \begin{tikzpicture}
   \tikzset{every node/.style=uStyle}
   \begin{scope}[xshift=0in, yshift=-0.25in, scale=.5]
   \node (a) at (-4,0) {}; 
   \node (b) at (-2,0) {};
   \node (c) at (0,0) {}; 
   \node (d) at (2,0) {}; 
   \node (e) at (4,0) {}; 
   \draw[thick, black] (a)--(b)--(c)--(d)--(e);
   \node[lStyle, inner sep=0pt] (p5) at (0,-1.75) {\small$P_5$};
   \end{scope}
 
   \begin{scope}[xshift=1.8in, scale=.5, yshift=0in]
   \node (a) at (-2,0) {};
   \node (b) at (0,0) {}; 
   \node (c) at (2,0) {};
   \node (d) at (4,0) {};
   \node (e) at (1,-2) {};
   \draw[thick, black] (a)--(b)--(c)--(d) (a)--(e) (b)--(e) (c)--(e) (d)--(e);  
   
   \draw node[lStyle, inner sep=0pt] at ([shift=({-90:.4in})]e) {\small$\text{gem}$};
   \end{scope}
  \end{tikzpicture}
\caption{A $P_5$ and a gem.}
 \label{fig:fig2}
\end{figure}
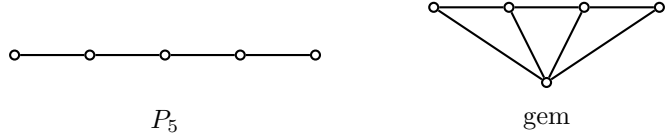

\indent We use the following notation from
~\cite{CKMM19}. Let $G$ be a graph and $X, Y\subseteq V(G)$. Let
$[X, Y ]$ denote the set of edges that have one end in $X$ and other end in
$Y$. If every vertex in $X$ is adjacent to every vertex in $Y$, then $X$
is \textit{complete} to $Y$ or $[X, Y ]$ is complete; if $[X, Y]=\emptyset$,
then $X$ is \textit{anticomplete} to $Y$. A set $X$ is a \textit{homogeneous
set} if every vertex with a neighbor in $X$ is complete to $X$. 

An \emph{expansion} of a graph $H$ is any graph $G$ such that $V(G)$ can be
partitioned into $\card{V(H)}$ non-empty sets $Q_v$, for each $v\in
V(H)$, such that $[Q_u,Q_v]$ is complete if $uv \in E(H)$, and
$[Q_u,Q_v]$ is anticomplete if $uv\not\in E(H)$. An expansion of a
graph is a \textit{clique expansion} if each $Q_v$ is a clique, and is a
\textit{$P_4$-free expansion} if each $Q_v$ induces a $P_4$-free graph.

For each $i\in\{1,\ldots,10\}$, the class of graphs $\G_i$ is all the $P_4$-free
expansions of the single graph shown in Case $i$ below (in
Figures~\ref{case123-fig}--\ref{case910H-fig}), where each $Q_i$ is a single
vertex.  The class $\G_i^*$ is all clique expansions of the graph in Case $i$.
The classes $\HH$ and $\HH^*$ are defined nearly analogously for the graph in
Case~11.  (There is a slight difference, which we explain when we
handle Case~11.)

%%%%%%%%%%%%%%%%%%%%%%%%%%%%%%%%%%%%%%%%%%%%%%%%%%%%
%%%%%%%%%%%%%%%                    Lemmas and Reductions              %%%%%%%%%%%%%%
%%%%%%%%%%%%%%%%%%%%%%%%%%%%%%%%%%%%%%%%%%%%%%%%%%%%

\section{Lemmas and Reductions}

In this section, we present lemmas that we will 
use to prove our Main Theorem.
We will assume the Main Theorem is false and choose $G$ to be a
counterexample that is \emph{vertex-critical}; that is $\chi(G-v)<\chi(G)$ for
each $v\in V(G)$.  Ultimately, we reach a contradiction, by constructing a
$(\Delta(G)-1)$-coloring of $G$.  To construct this coloring, we repeatedly use
the following easy lemma.

%Degeneracy Reduction.
\begin{lemma}
\label{degen-lem}
Fix $k\in \mathbb{Z}^+$.  Let $G$ be a graph and let $I_1, \dots, I_t$ be
pairwise disjoint independent sets of $G$. If $G-\bigcup_{j=1}^t
I_j$ is $(k-t-1)$-degenerate, then $\chi(G)\le k$.
\end{lemma}
\begin{proof}
Let $G$, $k$, and $I_1, \dots, I_t$ satisfy the hypotheses.  %Let $G':=$. 
If $G-\bigcup_{j=1}^t I_j$ is $(k-t-1)$-degenerate, then it has a
$(k-t)$-coloring $\varphi'$. We extend $\varphi'$ to a $k$-coloring of $G$ by
giving each $I_j$ its own color. 
\end{proof}

We will use Lemma~\ref{degen-lem} to $8$-color $G$, typically with $t=2$. 
Let $G':=G-\bigcup_{j=1}^2 I_j$.
To prove that $G'$ is 5-degenerate we give a vertex order
$\sigma=(v_1,\dots,v_n)$ such that each $v_i$ has at most 5 neighbours
earlier in $\sigma$. For a vertex partition $S_1\uplus \cdots\uplus
S_t$ of $V(G)$ we write $(S_1, \dots, S_t)$ to denote a vertex order
$\sigma$ where all vertices in $S_i$ come before all vertices in
$S_{i+1}$, for each $i$, and vertices within each
$S_i$ are ordered arbitrarily.

\begin{lemma}
\label{8-9-lem}
Every vertex-critical counterexample $G$ to the Borodin--Kostochka Conjecture
has $\delta(G)\ge \Delta(G)-1$.  In particular,
$\delta(G)\ge8$ and $\card{d(v)-d(w)}\leq 1$ for all $v,w\in V(G)$.
\begin{proof}
Assume that $G$ is a vertex-critical counterexample to the the Borodin--Kostochka
Conjecture.  Suppose there exists $v\in V(G)$ such that $d(v)<\Delta(G)-1$.
Since $G$ is vertex-critical, $G-v$ has a $(\Delta(G)-1)$-coloring $\vph$.
To extend $\vph$ to $v$, we simply color $v$ with a color unused on $N_G(v)$.
\end{proof}
\end{lemma}

We want to reuse the idea in the proof of Lemma~\ref{8-9-lem} to show that
other induced subgraphs cannot appear in $G$.  For this, we use notation and a
lemma from~\cite{CR13}.  For a graph $H$, a \emph{$d_1$-assignment} $L$ gives to each
$v\in V(H)$ a set (of allowable colors) $L(v)$ such that $|L(v)|=d_H(v)-1$. 
A graph $H$ is \emph{$d_1$-choosable} if $H$ has a proper coloring $\vph$ with
$\vph(v)\in L(v)$ for every $d_1$-assignment $L$.

\begin{lemma}[\cite{CR13}]
\label{d1-choosable-lem}
If $G$ is vertex-critical and $\chi(G)=\Delta(G)$, then $G$ cannot contain
any non-empty, $d_1$-choosable,
induced subgraph $H$.  So such a $G$ contains neither of the
following as induced subgraphs: 
\begin{itemize}
\item $K_3\vee 3K_2$ or
\item $K_4\vee H$, where $V(H)$ contains two disjoint pairs of nonadjacent vertices.
\end{itemize}
\end{lemma}
\begin{proof}
We begin with the first statement.  Suppose, to the contrary, that $G$ contains
an induced subgraph $H$, where $G$ and $H$ satisfy the hypotheses. 
Since $G$ is vertex-critical, $G-H$ has a $(\Delta(G)-1)$-coloring $\vph$.
For each $v\in V(H)$, let $L(v)$ denote the
colors in $\{1,\ldots,\Delta(G)-1\}$ that are unused by $\vph$ on $N_G(v)$.  Thus,
$|L(v)|\ge \Delta(G)-1-(d_G(v)-d_H(v))\ge d_H(v)-1$.  Now we can extend $\vph$
to $H$ precisely because $H$ is $d_1$-choosable.  This gives a
$(\Delta(G)-1)$-coloring of $G$, which is a contradiction; this proves the first
statement.
The second statement follows from \cite[][Lemmas 3.3 \&
3.10]{CR13}, where it is proved that the two subgraphs listed above are $d_1$-choosable. 
\end{proof}

%"Copycat" lemma
\begin{copycatlem}
\label{copycat}
Let $G$ be a vertex-critical counterexample to the Borodin--Kostochka
Conjecture. If $G$ contains nonempty, disjoint homogeneous sets $A$ and $B$
such that $A$ and $B$ are cliques and $N(A)\subseteq N(B)$, then $\card{A}
>\card{B}$.
\end{copycatlem}
\begin{proof}
Let $G$, $A$, and $B$ satisfy the hypotheses. Suppose, for the sake of
contradiction, that $\card{A}\leq \card{B}$. Since $G$ is vertex-critical,
$G- A$ has a $(\Delta(G)-1)$-coloring $\varphi$. Since $\card{A}\leq
\card{B}$, $N(A)\subseteq N(B)$, and $B$ is a homogeneous set, we can extend
$\vph$ to $G$ by coloring $A$ with colors used on $B$.  Thus, $G$ is not a
counterexample to the Borodin--Kostochka Conjecture, a contradiction.
\end{proof}

%Kostochka's Reduction that we may assume Delta=9.
A graph class $\G$ is \textit{hereditary} if for each $G\in\G$ we have $H\in
\G$ for each induced subgraph $H$ of $G$. Every class of graphs characterized
by a list of forbidden induced subgraphs is a hereditary class; in particular,
the class of $(P_5, \text{gem})$-free graphs is hereditary.

\begin{theorem}[Kostochka \cite{Ko}, Catlin \cite{Ca}] \label{catlin-reduc-lem}
Let $\G$ be a hereditary class of graphs.  If the Borodin--Kostochka Conjecture
is false for some $G\in \G$, then it is false for some $G\in \G$ with $\Delta(G)=9$.
\end{theorem}

\begin{proof}
We assume the Borodin--Kostochka Conjecture is true for all $G\in\mathcal{G}$
with $\Delta(G)=9$ and show that it is true for all $G\in \mathcal{G}$.
Suppose instead that some $G\in \mathcal{G}$ with $\Delta(G)>9$
is a counterexample; in particular, $\omega(G)\le \Delta(G)-1$.  Among all such
$G$, choose one to minimize $\Delta(G)$.
We want to find a maximum independent set $I$ that intersects each clique of
size $\Delta(G)-1$.  If $\omega(G)<\Delta(G)-1$, then any
maximum independent set $I$ suffices.  Otherwise, $I$ is guaranteed by a
result of King~\cite{King}. Let $G':=G-I$.  
Note that $\omega(G')\le \Delta(G)-2$ and $\Delta(G')\le \Delta(G)-1$, since $I$
is maximum.

If $\Delta(G')\le \Delta(G)-3$, then we can greedily color $G'$ with at most
$\Delta(G)-2$ colors.  
By using a new color on $I$, we get a $(\Delta(G)-1)$-coloring of $G$, a contradiction.
If $\Delta(G')=\Delta(G)-2$, then $G'$ has a $(\Delta(G)-2)$-coloring by Brooks'
Theorem, since $\omega(G')\le \Delta(G)-2$.  By using a new color on $I$, we
again get a $(\Delta(G)-1)$-coloring of $G$, a contradiction.
So we must have $\Delta(G')=\Delta(G)-1$ and $\omega(G')\le \Delta(G)-2$.
Since $G'\in \mathcal{G}$ and $\Delta(G')<\Delta(G)$, we know that $G'$ is not a
counterexample to the Borodin--Kostochka Conjecture.
In particular $G'$ has a $(\Delta(G')-1)$-coloring.
By using a new color on $I$, we extend this coloring of $G'$ to a
$(\Delta(G)-1)$-coloring of $G$, a contradiction.
\end{proof}

We conclude this section with two results of Chudnovsky, Karthick, Maceli, and Maffray.

\begin{theorem}[{\cite[Theorem 3]{CKMM19}}]
\label{CKMM-thm1}
If $G$ is a connected $(P_5, \text{gem})$-free graph that contains an induced
$C_5$, then either $G\in \mathcal{H}$ or $G\in \G_i$, for some
$i\in\{1,\ldots,10\}$. 
\label{11-cases-thm}
\end{theorem}

\begin{theorem}[\cite{CKMM19,CKMM18}]
\label{CKMM-thm2}
Fix $i\in \{1,\dots, 10\}$. For every $G\in\mathcal{G}_i$ (resp.
$G\in\mathcal{H}$) there is $G^*\in\mathcal{G}_i^*$ (resp.
$G^*\in\mathcal{H}^*$) such that $\omega(G) = \omega(G^*)$ and $\chi(G) =
\chi(G^*)$. Further, $G^*$ is an induced subgraph of $G$. 
\label{Gi*-thm}
\end{theorem}

\section{Proof of Main Result}

\begin{mainthm} 
Let $G$ be a $(P_5,\text{gem})$-free graph.
If $\Delta(G)\geq 9$ and $\omega(G)\leq \Delta(G) -1$, then $\chi(G)\leq \Delta(G)-1$.
\end{mainthm}

\begin{proof} 
Suppose the theorem is false.  Let $G$ be a counterexample that minimizes $\Delta(G)$.
Further, we can choose $G$ to be vertex-critical.  By Theorem~\ref{catlin-reduc-lem}, 
we can assume that $\Delta(G)=9$.
Next we show that we can also assume that either $G\in \G^*_i$ for some
$i\in\{1,\ldots,10\}$ or $G\in \HH^*$.

If $G$ is perfect, then $\chi(G)=\omega(G)\le \Delta(G)-1$, a contradiction.
Since $G$ is not perfect, the Strong Perfect Graph Theorem \cite{SPGT} 
implies that $G$ must contain an odd hole or odd anti-hole (that is, either $G$
or its complement contains an induced odd cycle of length at least 5).
Every odd hole of length at least 7 contains a $P_5$ as an induced subgraph.
Similarly, every odd antihole of length at least 7 contains a gem as an induced
subgraph.  Since $G$ is $(P_5,\textrm{gem})$-free, $G$ must contain a hole or
antihole of length 5.  In fact, these are both congruent to $C_5$.  So, by
Theorem~\ref{11-cases-thm}, either $G\in G_i$ for some $i\in\{1,\ldots,10\}$ or
$G\in \HH$.  Finally, by Theorem~\ref{Gi*-thm}, we can assume that either $G\in
\G_i^*$ or $G\in \HH^*$.  This is because $G$ is vertex-critical.  Since
$\chi(G^*)=\chi(G)$, we conclude that $G^*\cong G$.
For each $Q_i$, we write $x_i$, $x_i'$, $x_i''$, and $x_i'''$ to denote
arbitrary distinct vertices in $Q_i$ (provided that such vertices exist).
For independent sets $I_1$ and $I_2$ and $G-(I_1\cup I_2)$, we write
$(Q'_{j_1},Q'_{j_2},Q'_{j_3},\ldots)$ to denote the vertex order 
$(Q_{j_1},Q_{j_2},Q_{j_3},\ldots)$ restricted to $V(G)\setminus (I_1\cup I_2)$.
Now we consider these 11 cases in succession.  Each case is independent of all
others.

%%%%%%%%%%%%%%%%%%%%%%%%%%%%%%%%%%%%%%%%%%%%%%%%%%%%%%%        G_1
\textbf{Case 1: $\bm{G\in\mathcal{G}_1^*}$.}  
Since $\Delta(G)=9$, there exists $i\in\{1,\dots,5\}$ such that
$\card{Q_i}\geq 4$. By symmetry, assume $i=1$. If $\card{Q_5}\geq 2$ and
$\card{Q_2}\geq 2$, then $G[\{x_1, x_1', x_1'', x_1''', x_2, x_2', x_5, x_5'\}]
\cong K_4\vee H$, where $H$ contains two disjoint pairs of nonadjacent
vertices; this subgraph is $d_1$-choosable, which contradicts
Lemma~\ref{d1-choosable-lem}.  Assume instead, by symmetry, that
$\card{Q_5}=1$.  %Since $d(x_1)\geq 8$, we know $\card{Q_2}\geq 4$. 
Lemma~\ref{8-9-lem} implies that
$1\ge d(x_2)-d(x_1)=|N[x_2]|-|N[x_1]|=|Q_3|-|Q_5|$.
%$9\ge d(x_2)=\card{Q_1}+\card{Q_2}-1+\card{Q_3}$, we know 
Thus, $\card{Q_3} \leq 2$. Since $8\le d(x_4) = \card{Q_3}+ \card{Q_4}-1
+\card{Q_5}$, we know $\card{Q_4}\geq 6$. But now $d(x_5)=
\card{Q_4} + (\card{Q_5}-1) + \card{Q_1}\ge 6+(1-1)+4=10$.
This contradicts that $\Delta(G)=9$.

\begin{figure}[!b]
\centering
\begin{tikzpicture}[scale=.5]
\begin{scope}%G_1
\node[draw=black, fill=white, thick, circle, inner sep=1.2pt] (1) at (0,.85)
{\footnotesize$Q_1$}; 
\node[draw=black, fill=white, thick, circle, inner sep=1.2pt] (2) at
([shift=({-36:1in})]1) {\footnotesize$Q_2$};
\node[draw=black, fill=white, thick, circle, inner sep=1.2pt] (3) at
([shift=({-108:1in})]2) {\footnotesize$Q_3$};
\node[draw=black, fill=white, thick, circle, inner sep=1.2pt] (4) at
([shift=({-180:1in})]3) {\footnotesize$Q_4$};
\node[draw=black, fill=white, thick, circle, inner sep=1.2pt] (5) at
([shift=({-252:1in})]4) {\footnotesize$Q_5$};
\draw[line width=2pt, black] (1)--(2)--(3)--(4)--(5)--(1); \node at (0,-4.2)
{\footnotesize Case 1: $G\in \mathcal{G}_1^*$}; 
\end{scope}

\begin{scope}[xshift=3in]%G_2
\node[draw=black, fill=white, thick, circle, inner sep=1.2pt] (1) at (0,.85)
{\footnotesize$Q_1$}; 
\node[draw=black, fill=white, thick, circle, inner sep=1.2pt] (2) at
([shift=({-36:1in})]1) {\footnotesize$Q_2$};
\node[draw=black, fill=white, thick, circle, inner sep=1.2pt] (3) at
([shift=({-108:1in})]2) {\footnotesize$Q_3$};
\node[draw=black, fill=white, thick, circle, inner sep=1.2pt] (4) at
([shift=({-180:1in})]3) {\footnotesize$Q_4$};
\node[draw=black, fill=white, thick, circle, inner sep=1.2pt] (5) at
([shift=({-252:1in})]4) {\footnotesize$Q_5$};
\node[draw=black, fill=white, thick, circle, inner sep=1.2pt] (6) at (0,-1.2)
{\footnotesize$Q_6$};
\draw[line width=2pt, black] (1)--(2)--(3)--(4)--(5)--(1) (1)--(6)--(4)
(6)--(3); \node at (0,-4.2) {\footnotesize Case 2: $G\in \mathcal{G}_2^*$}; 
\end{scope} 

\begin{scope}[xshift=6in]%G_3
\node[draw=black, fill=white, thick, circle, inner sep=1.2pt] (1) at (0,.85)
{\footnotesize$Q_1$}; 
\node[draw=black, fill=white, thick, circle, inner sep=1.2pt] (2) at
([shift=({-36:1in})]1) {\footnotesize$Q_2$};
\node[draw=black, fill=white, thick, circle, inner sep=1.2pt] (7) at
([shift=({-108:1in})]2) {\footnotesize$Q_3$};
\node[draw=black, fill=white, thick, circle, inner sep=1.2pt] (4) at
([shift=({-180:1in})]7) {\footnotesize$Q_4$};
\node[draw=black, fill=white, thick, circle, inner sep=1.2pt] (5) at
([shift=({-252:1in})]4) {\footnotesize$Q_5$};
\node[draw=black, fill=white, thick, circle, inner sep=1.2pt] (3) at
([shift=({36:.9in})]4)  {\footnotesize$Q_7$};
\node[draw=black, fill=white, thick, circle, inner sep=1.2pt] (6) at
([shift=({72:.9in})]4) {\footnotesize$Q_6$};
\draw[line width=2pt, black] (1)--(2)--(7)--(4)--(5)--(1) (4)--(6)--(1)
(4)--(3)--(2) (6)--(3); \node at (0,-4.2) {\footnotesize Case 3: $G\in \mathcal{G}_3^*$}; 
\end{scope}
\end{tikzpicture}
\caption{Cases 1--3: $G\in \G_1^*,G\in G_2^*$, or $G\in G_3^*$.\label{case123-fig}}
\end{figure}

%%%%%%%%%%%%%%%%%%%%%%%%%%%%%%%%%%%%%%%%%%%%%%%%%%%%%%%        G_2
\textbf{Case 2: $\bm{G\in\mathcal{G}_2^*}$.}
By the Copycat Lemma, $\card{Q_2}>\card{Q_6}$ and $\card{Q_5}>\card{Q_6}$;
in particular, $\card{Q_2}, \card{Q_5}\geq 2$ since each  $Q_i$ is nonempty.
If $\card{Q_6}\geq 2$, then let %$G':= G- (I_1\cup I_2)$, where 
$I_1 = \{x_2, x_5, x_6\}$ and $I_2=\{x_2',x_5',x_6'\}$. Now $G-(I_1\cup I_2)$
is $5$-degenerate with vertex order $(Q_1',Q_4', Q_3', Q_5',Q_2',Q_6')$; by
Lemma~\ref{degen-lem}, $G$ is 8-colorable, a contradiction.  Assume instead
that $\card{Q_6}=1$.  If $\card{Q_1}\geq 2$, then let %$G':= G- I_1$, where 
$I_1=\{x_2,x_5,x_6\}$. Now $G-I_1$ is $6$-degenerate with vertex order $(Q_1',
Q_5', Q_2', Q_4', Q_3', Q_6')$; by Lemma~\ref{degen-lem}, $G$ is 8-colorable, a
contradiction. Assume instead that $\card{Q_1}=1$.  Now Lemma~\ref{8-9-lem}
implies that $1\geq d(x_3)-d(x_2)=\card{N[x_3]}-\card{N[x_2]} = \card{Q_4} +
\card{Q_6} - \card{Q_1} = \card{Q_4}$. By symmetry, $\card{Q_3}\leq 1$. Thus,
$d(x_6)=\card{Q_1}+\card{Q_3}+\card{Q_4}+\card{Q_6}-1=3$, which contradicts
that $\delta(G)\ge 8$.
\smallskip

%%%%%%%%%%%%%%%%%%%%%%%%%%%%%%%%%%%%%%%%%%%%%%%%%%%%%%%        G_3
\textbf{Case 3: $\bm{G\in\mathcal{G}_3^*}$.}
By the Copycat Lemma, $\card{Q_5}>\card{Q_6}$ and $\card{Q_3}>\card{Q_7}$;
in particular,  $\card{Q_5}, \card{Q_3}\geq 2$, since each  $Q_i$ is nonempty.
If $\card{Q_4}\geq 2$, then let %$G':= G- (I_1\cup I_2)$, where
$I_1=\{x_2,x_5,x_6\}$ and $I_2=\{x_1,x_3,x_7\}$. Now $G-(I_1\cup I_2)$ is
$5$-degenerate with vertex order $(Q_4', Q_3', Q_5', Q_2', Q_1', Q_6',
Q_7')$; by Lemma~\ref{degen-lem}, $G$ is 8-colorable, a contradiction.  Assume
instead that $\card{Q_4}=1$.
Lemma~\ref{8-9-lem} implies that $1\geq d(x_2)-d(x_3) = \card{N[x_2]}-\card{N[x_3]} =
\card{Q_1}+\card{Q_7}-\card{Q_4}=\card{Q_1}+\card{Q_7}-1$. 
Thus, $\card{Q_1} = \card{Q_7}=1$. 
By symmetry, $\card{Q_2}=\card{Q_6}=1$. So
$d(x_6) = \card{Q_1}+\card{Q_7}+\card{Q_4}+\card{Q_6}-1=3$, which contradicts
that $\delta(G)\ge 8$.

%%%%%%%%%%%%%%%%%%%%%%%%%%%%%%%%%%%%%%%%%%%%%%%%%%%%%%%        G_4     
\textbf{Case 4: $\bm{G\in\mathcal{G}_4^*}$.} By the Copycat Lemma,
$\card{Q_1}>\card{Q_5}$; so $\card{Q_1}\geq 2$ since each $Q_i$ is
nonempty.  Let $I_1:=\{x_1,x_5,x_7\}$ and $I_2:=\{x_1',x_3,x_6\}$. Now
$G-(I_1\cup I_2)$ is $5$-degenerate with vertex order $(Q_4', Q_2', Q_1', Q_5',
Q_3', Q_7', Q_6')$; by Lemma~\ref{degen-lem}, $G$ is 8-colorable, a contradiction.

\begin{figure}[!h]
\centering
\begin{tikzpicture}[scale=.3] 

\begin{scope}[xshift=-6.5in, yshift=-1.2in,yscale=-1, scale=2]%G_4
\node[draw=black, fill=white, thick, circle, inner sep=1.2pt] (1) at (2,1.5)
{\footnotesize$Q_2$}; 
\node[draw=black, fill=white, thick, circle, inner sep=1.2pt] (6) at (-2,1.5)
{\footnotesize$Q_3$};
\node[draw=black, fill=white, thick, circle, inner sep=1.2pt] (4) at (-2,-1.5)
{\footnotesize$Q_4$};
\node[draw=black, fill=white, thick, circle, inner sep=1.2pt] (7) at (2,-1.5)
{\footnotesize$Q_5$};
\node[draw=black, fill=white, thick, circle, inner sep=1.2pt] (2) at
([shift=({-125:.75in})]1) {\footnotesize$Q_6$};
\node[draw=black, fill=white, thick, circle, inner sep=1.2pt] (3) at
([shift=({-55:.75in})]6)  {\footnotesize$Q_7$};
\node[draw=black, fill=white, thick, circle, inner sep=1.2pt] (5) at
([shift=({-45:.55in})]7) {\footnotesize$Q_1$};
\draw[line width=2pt, black] (1)--(6)--(4)--(7)--(1) (6)--(3)--(2)--(1) (2)--(7) (3)--(4);
\draw[black, line width=2pt] plot [smooth, tension=1] coordinates {(5) (3,0) (1)};
\draw[black, line width=2pt] plot [smooth, tension=1] coordinates {(5) (0,-2.5) (4)};
\node[draw=black, fill=white, thick, circle, inner sep=1.2pt] (1) at (2,1.5)
{\footnotesize$Q_2$}; 
\node[draw=black, fill=white, thick, circle, inner sep=1.2pt] (4) at (-2,-1.5)
{\footnotesize$Q_4$};
\node[draw=black, fill=white, thick, circle, inner sep=1.2pt] at
([shift=({-45:.55in})]7) {\footnotesize$Q_1$};

\node at (0,2.78) {\footnotesize Case 4: $G\in \mathcal{G}_4^*$}; 
\end{scope}

\begin{scope}%G_5
\node[draw=black, fill=white, thick, circle, inner sep=1.2pt] (1) at (0,1.76) {\footnotesize$Q_1$}; 
\node[draw=black, fill=white, thick, circle, inner sep=1.2pt] (8) at ([shift=({-36:2in})]1) {\footnotesize$Q_2$};
\node[draw=black, fill=white, thick, circle, inner sep=1.2pt] (3) at ([shift=({-108:2in})]8) {\footnotesize$Q_3$};
\node[draw=black, fill=white, thick, circle, inner sep=1.2pt] (4) at ([shift=({-180:2in})]3) {\footnotesize$Q_4$};
\node[draw=black, fill=white, thick, circle, inner sep=1.2pt] (5) at ([shift=({-252:2in})]4) {\footnotesize$Q_5$};

\node[draw=black, fill=white, thick, circle, inner sep=1.2pt] (2) at ([shift=({-60:1.5in})]1)  {\footnotesize$Q_7$};
\node[draw=black, fill=white, thick, circle, inner sep=1.2pt] (6) at ([shift=({-90:2in})]1) {\footnotesize$Q_8$};
\node[draw=black, fill=white, thick, circle, inner sep=1.2pt] (7) at ([shift=({-120:1.5in})]1) {\footnotesize$Q_6$};
\draw[line width=2pt, black] (1)--(8)--(3)--(4)--(5)--(1) (4)--(6)--(1)  (6)--(3) (4)--(7)--(2)--(3) (7)--(1)--(2);
\node at (0,-8.5) {\footnotesize Case 5: $G\in \mathcal{G}_5^*$}; 
\end{scope}

\begin{scope}[scale=1.45, xshift=4.3in, yshift=-.86in,yscale=-1]%G_6
\node[draw=black, fill=white, thick, circle, inner sep=1.2pt] (4) at (-2.75,-3) {\footnotesize$Q_4$};
\node[draw=black, fill=white, thick, circle, inner sep=1.2pt] (5) at (-2.75,2) {\footnotesize$Q_3$}; 
\node[draw=black, fill=white, thick, circle, inner sep=1.2pt] (8) at (2.75,2) {\footnotesize$Q_2$};
\node[draw=black, fill=white, thick, circle, inner sep=1.2pt] (2) at (2.75,-3) {\footnotesize$Q_1$};

\node[draw=black, fill=white, thick, circle, inner sep=1.2pt] (1) at ([shift=({-17:1.75in})]5) {\footnotesize$Q_6$};
\node[draw=black, fill=white, thick, circle, inner sep=1.2pt] (6) at ([shift=({-163:1.75in})]8)  {\footnotesize$Q_7$};
\node[draw=black, fill=white, thick, circle, inner sep=1.2pt] (3) at ([shift=({-80:.75in})]6) {\footnotesize$Q_8$};
\node[draw=black, fill=white, thick, circle, inner sep=1.2pt] (7) at ([shift=({-100:.75in})]1) {\footnotesize$Q_5$};
\draw[line width=2pt, black] (4)--(5)--(8)--(2) (5)--(1)--(7)--(2)--(3)--(4)--(6)--(1) (8)--(6)--(3) (2)--(1) (7)--(4);
\node at (0,3.7) {\footnotesize Case 6: $G\in \mathcal{G}_6^*$}; 
\end{scope}
\end{tikzpicture}
\caption{Cases 4--6: $G\in \G_4^*,G\in G_5^*$, or $G\in G_6^*$.\label{case456-fig}}
\end{figure}

%%%%%%%%%%%%%%%%%%%%%%%%%%%%%%%%%%%%%%%%%%%%%%%%%%%%%%%        G_5
\textbf{Case 5: $\bm{G\in\mathcal{G}_5^*}$.}
By the Copycat Lemma, $\card{Q_5}>\card{Q_6}$ and $\card{Q_2}>\card{Q_7}$;
so $\card{Q_5}, \card{Q_2}\geq 2$ since each  $Q_i$ is nonempty.
If $\card{Q_1}\geq 3$, then $G[\{x_1,x_1',x_1'',x_2,x_2',
x_5,x_5',x_6,x_7\}]\cong K_3\vee3K_2$; this subgraph is $d_1$-choosable, 
contradicting Lemma~\ref{d1-choosable-lem}.  Assume instead that $\card{Q_1}\leq2$. 
If $\card{Q_3}\geq 4$, then $G[\{x_3,x_3',x_3'',x_3''',x_2, x_4, x_7,
x_8\}]\cong K_4\vee H$ where $H$ contains two disjoint pairs of nonadjacent
vertices; this subgraph is $d_1$-choosable, contradicting
Lemma~\ref{d1-choosable-lem}. So $\card{Q_3}\leq3$. By symmetry, $\card{Q_4}\leq 3$. 
Thus, $\card{Q_2},\card{Q_5}\geq 4$, since $d(x_2),d(x_5)\geq 8$. However, now
$d(x_1)=(\card{Q_1}-1)+\card{Q_2}+\card{Q_5}+\card{Q_6}+\card{Q_7}\geq
(1-1)+4+4+1+1=10$, which contradicts that $\Delta(G)=9$.

%%%%%%%%%%%%%%%%%%%%%%%%%%%%%%%%%%%%%%%%%%%%%%%%%%%%%%%        G_6
\textbf{Case 6: $\bm{G\in\mathcal{G}_6^*}$.}
%Let $G':= G- (I_1\cup I_2)$, where 
Let $I_1:=\{x_3,x_5,x_7\}$ and
$I_2:=\{x_2,x_6,x_8\}$. Now $G-(I_1\cup I_2)$ is $5$-degenerate with vertex order $(Q_4',
Q_1', Q_8',Q_5', Q_7', Q_6', Q_3', Q_2')$.  By Lemma~\ref{degen-lem}, $G$ is 8-colorable, a contradiction.

\begin{figure}[!h]
\centering
\begin{tikzpicture}[yscale=.9]
\begin{scope}[rotate=90, yscale=-1, scale=.55] %G_7
\node[draw=black, fill=white, thick, circle, inner sep=1.2pt] (5) at (2,1.5) {\footnotesize$Q_5$}; 
\node[draw=black, fill=white, thick, circle, inner sep=1.2pt] (2) at (-2,1.5) {\footnotesize$Q_2$};
\node[draw=black, fill=white, thick, circle, inner sep=1.2pt] (1) at (-2,-1.5) {\footnotesize$Q_3$};
\node[draw=black, fill=white, thick, circle, inner sep=1.2pt] (4) at (2,-1.5) {\footnotesize$Q_7$};

\node[draw=black, fill=white, thick, circle, inner sep=1.2pt] (6) at ([shift=({-125:.75in})]5) {\footnotesize$Q_8$};
\node[draw=black, fill=white, thick, circle, inner sep=1.2pt] (3) at ([shift=({-55:.75in})]2)  {\footnotesize$Q_6$};
\node[draw=black, fill=white, thick, circle, inner sep=1.2pt] (7) at ([shift=({-45:.55in})]4) {\footnotesize$Q_4$};
\node[draw=black, fill=white, thick, circle, inner sep=1.2pt] (8) at ([shift=({45:.55in})]5) {\footnotesize$Q_1$};
\draw[line width=2pt, black] (5)--(2)--(1)--(4)--(5) (2)--(3)--(6)--(5) (6)--(4) (3)--(1);

\draw[black, line width=2pt] plot [smooth, tension=1] coordinates {(7) (3,0) (5)};
\draw[black, line width=2pt] plot [smooth, tension=1] coordinates {(7) (0,-2.5) (1)};

\draw[black, line width=2pt] plot [smooth, tension=1] coordinates {(8) (3,0) (4)};
\draw[black, line width=2pt] plot [smooth, tension=1] coordinates {(8) (0, 2.5) (2)};

\node[draw=black, fill=white, thick, circle, inner sep=1.2pt] (5) at (2,1.5) {\footnotesize$Q_5$}; 
\node[draw=black, fill=white, thick, circle, inner sep=1.2pt] (1) at (-2,-1.5) {\footnotesize$Q_3$};
\node[draw=black, fill=white, thick, circle, inner sep=1.2pt] at ([shift=({-45:.55in})]4) {\footnotesize$Q_4$};

\node[draw=black, fill=white, thick, circle, inner sep=1.2pt] (2) at (-2,1.5) {\footnotesize$Q_2$};
\node[draw=black, fill=white, thick, circle, inner sep=1.2pt] (4) at (2,-1.5) {\footnotesize$Q_7$};
\node[draw=black, fill=white, thick, circle, inner sep=1.2pt] (8) at ([shift=({45:.55in})]5) {\footnotesize$Q_1$};
\node at (-3.1,0) {\footnotesize Case 7: $G\in \mathcal{G}_7^*$}; 
\end{scope}

\begin{scope}[xshift=2.5in, yshift=.4in, scale=.8] %G_8
\node[fill=white, circle, inner sep=1.2pt] (0) at (0,0.85) {}; 
\node[draw=black, fill=white, thick, circle, inner sep=1.2pt] (6) at ([shift=({0:0.3in})]0) {\footnotesize$Q_6$}; 
\node[draw=black, fill=white, thick, circle, inner sep=1.2pt] (5) at ([shift=({180:0.3in})]0) {\footnotesize$Q_5$}; 
\node[draw=black, fill=white, thick, circle, inner sep=1.2pt] (1) at ([shift=({-30:.75in})]0) {\footnotesize$Q_1$};
\node[draw=black, fill=white, thick, circle, inner sep=1.2pt] (2) at ([shift=({-100:1in})]1) {\footnotesize$Q_2$};
\node[draw=black, fill=white, thick, circle, inner sep=1.2pt] (4) at ([shift=({-150:.75in})]0) {\footnotesize$Q_4$};
\node[draw=black, fill=white, thick, circle, inner sep=1.2pt] (3) at ([shift=({-80:1in})]4) {\footnotesize$Q_3$};
\node[draw=black, fill=white, thick, circle, inner sep=1.2pt] (8) at ([shift=({-90:0.90in})]6) {\footnotesize$Q_8$};
\node[draw=black, fill=white, thick, circle, inner sep=1.2pt] (7) at ([shift=({-90:0.90in})]5) {\footnotesize$Q_7$};
\draw[line width=2pt, black] (6)--(4)--(3)--(2)--(1)--(5)--(4)--(8)--(2) (3)--(7)--(1)--(6)--(8) (5)--(7);
\draw[black, line width=2pt] plot [smooth, tension=1] coordinates {(7) (0,-2) (8)};
\node[draw=black, fill=white, thick, circle, inner sep=1.2pt] (8) at ([shift=({-90:0.9in})]6) {\footnotesize$Q_7$};
\node[draw=black, fill=white, thick, circle, inner sep=1.2pt] (7) at ([shift=({-90:0.9in})]5) {\footnotesize$Q_8$};
\node at (0,-3.3) {\footnotesize Case 8: $G\in \mathcal{G}_8^*$}; 
\end{scope}
\end{tikzpicture}
\caption{Cases 7--8: $G\in \G_7^*$ or $G\in G_8^*$.\label{case78-fig}}
\end{figure}

%%%%%%%%%%%%%%%%%%%%%%%%%%%%%%%%%%%%%%%%%%%%%%%%%%%%%%%        G_7
\textbf{Case 7: $\bm{G\in\mathcal{G}_7^*}$.}
By the Copycat Lemma, $\card{Q_4}>\card{Q_7}$; so
$\card{Q_4}\geq 2$ since each $Q_i$ is nonempty.
If $\card{Q_1}\geq 7$, then $d(x_2)=\card{Q_3} + (\card{Q_2}-1) + \card{Q_6} +
\card{Q_5} + \card{Q_1}\geq 10$, contradicting that $\Delta(G)=9$. Thus, $\card{Q_1}\leq 6$.
%Let $G':= G- (I_1\cup I_2)$, where 
Let $I_1:=\{x_4,x_6,x_7\}$ and $I_2 := \{x_2,x_4',x_8\}$. Now $G-(I_1\cup I_2)$
is $5$-degenerate with vertex order $(Q_5', Q_3', Q_4', Q_7', Q_2', Q_8', Q_6',
Q_1')$.  By Lemma~\ref{degen-lem}, $G$ is 8-colorable, a contradiction. 

%%%%%%%%%%%%%%%%%%%%%%%%%%%%%%%%%%%%%%%%%%%%%%%%%%%%%%%        G_8
\textbf{Case 8: $\bm{G\in\mathcal{G}_8^*}$.}
%Let $G':= G- (I_1\cup I_2)$, 
Let $I_1:=\{x_3,x_5,x_7\}$ and
$I_2:=\{x_2,x_6,x_8\}$. Now $G-(I_1\cup I_2)$ is $5$-degenerate with
order $(Q_4', Q_1', Q_3', Q_2', Q_7', Q_8', Q_6', Q_5')$; by
Lemma~\ref{degen-lem}, $G$ is 8-colorable, a contradiction.

%%%%%%%%%%%%%%%%%%%%%%%%%%%%%%%%%%%%%%%%%%%%%%%%%%%%%%%        G_9
\textbf{Case 9: $\bm{G\in\mathcal{G}_9^*}$.}
By the Copycat Lemma $\card{Q_9}>\card{Q_5}$; so $\card{Q_9}\geq 2$
since each $Q_i$ is nonempty.  %Let $G':= G- (I_1\cup I_2)$, 
Let $I_1 := \{x_3, x_5, x_7, x_9\}$ and $I_2:=\{x_2,x_6,x_8,x_9'\}$.  
Now $G-(I_1\cup I_2)$ is $5$-degenerate with order $(Q_4', Q_1', Q_3', Q_2',
Q_7', Q_8', Q_9', Q_6', Q_5')$.  By Lemma~\ref{degen-lem}, $G$ is 8-colorable,
a contradiction.

%%%%%%%%%%%%%%%%%%%%%%%%%%%%%%%%%%%%%%%%%%%%%%%%%%%%%%%        G_10
\textbf{Case 10: $\bm{G\in\mathcal{G}_{10}^*}$.}
%Let $G':= G- (I_1\cup I_2)$, where 
Let $I_1:=\{x_3,x_5,x_7\}$ and $I_2:=\{x_2,x_6,x_8\}$. Now $G-(I_1\cup I_2)$ is
$5$-degenerate with order $(Q_9', Q_4', Q_1', Q_3', Q_2', Q_7', Q_8', Q_6',
Q_5')$.  By Lemma~\ref{degen-lem}, $G$ is 8-colorable, a contradiction.

\begin{figure}[!h]
\centering
\begin{tikzpicture}
\begin{scope}[scale=.8, yshift=-.0975in] %G_9
\node[fill=white, circle, inner sep=1.2pt] (0) at (0,0.85) {}; 
\node[draw=black, fill=white, thick, circle, inner sep=1.2pt] (6) at ([shift=({0:0.3in})]0) {\footnotesize$Q_6$}; 
\node[draw=black, fill=white, thick, circle, inner sep=1.2pt] (5) at ([shift=({180:0.3in})]0) {\footnotesize$Q_5$}; 
\node[draw=black, fill=white, thick, circle, inner sep=1.2pt] (1) at ([shift=({-30:.75in})]0) {\footnotesize$Q_1$};
\node[draw=black, fill=white, thick, circle, inner sep=1.2pt] (2) at ([shift=({-100:1in})]1) {\footnotesize$Q_2$};
\node[draw=black, fill=white, thick, circle, inner sep=1.2pt] (4) at ([shift=({-150:.75in})]0) {\footnotesize$Q_4$};
\node[draw=black, fill=white, thick, circle, inner sep=1.2pt] (3) at ([shift=({-80:1in})]4) {\footnotesize$Q_3$};
\node[draw=black, fill=white, thick, circle, inner sep=1.2pt] (8) at ([shift=({-90:0.90in})]6) {\footnotesize$Q_7$};
\node[draw=black, fill=white, thick, circle, inner sep=1.2pt] (7) at ([shift=({-90:0.90in})]5) {\footnotesize$Q_8$};
\node[draw=black, fill=white, thick, circle, inner sep=1.2pt] (9) at ([shift=({180:0.65in})]1) {\footnotesize$Q_9$};
\draw[line width=2pt, black] (6)--(4)--(3)--(2)--(1)--(5)--(4)--(8)--(2) (3)--(7)--(1)--(6)--(8) (5)--(7) (1)--(9)--(4);
\draw[black, line width=2pt] plot [smooth, tension=1] coordinates {(7) (0,-2) (8)};
\node[draw=black, fill=white, thick, circle, inner sep=1.2pt] (8) at ([shift=({-90:0.9in})]6) {\footnotesize$Q_7$};
\node[draw=black, fill=white, thick, circle, inner sep=1.2pt] (7) at ([shift=({-90:0.9in})]5) {\footnotesize$Q_8$};
\node at (0,-3.5) {\footnotesize Case 9: $G\in \mathcal{G}_9^*$}; 
\end{scope}

\begin{scope}[xshift=2in, scale=.8] %G_10
\node[fill=white, circle, inner sep=0pt] (0) at (0,0.85) {}; 
\node[draw=black, fill=white, thick, circle, inner sep=1.2pt] (6) at ([shift=({0:0.25in})]0) {\footnotesize$Q_6$}; 
\node[draw=black, fill=white, thick, circle, inner sep=1.2pt] (5) at ([shift=({180:0.25in})]0) {\footnotesize$Q_5$}; 
\node[draw=black, fill=white, thick, circle, inner sep=1.2pt] (1) at ([shift=({-30:1in})]0) {\footnotesize$Q_1$}; 
\node[draw=black, fill=white, thick, circle, inner sep=1.2pt] (2) at ([shift=({-108:1in})]1) {\footnotesize$Q_2$}; 
\node[draw=black, fill=white, thick, circle, inner sep=1.2pt] (4) at ([shift=({-150:1in})]0) {\footnotesize$Q_4$}; 
\node[draw=black, fill=white, thick, circle, inner sep=1.2pt] (3) at ([shift=({-72:1in})]4) {\footnotesize$Q_3$}; 
\node[draw=black, fill=white, thick, circle, inner sep=1.2pt] (8) at ([shift=({-80:.9in})]6) {\footnotesize$Q_7$}; 
\node[draw=black, fill=white, thick, circle, inner sep=1.2pt] (7) at ([shift=({-100:.9in})]5) {\footnotesize$Q_8$}; 
\node[draw=black, fill=white, thick, circle, inner sep=1.2pt] (9) at ([shift=({-90:1.2in})]0) {\footnotesize$Q_9$}; 
\draw[line width=2pt, black] (6)--(4)--(3)--(2)--(1)--(5)--(4)--(8)--(2) (3)--(7)--(1)--(6)--(8) (5)--(7) (7)--(8) (5)--(9)--(6) (3)--(9)--(2);
\node at (0,-3.75) {\footnotesize Case 10: $G\in \mathcal{G}_{10}^*$}; 
\end{scope}

\begin{scope}[xshift=4in, yshift=.01in, scale=.75] %H (G_11)
\node[draw=black, fill=white, thick, circle, inner sep=1.2pt] (1) at (0,.85) {\footnotesize $A_1$}; 
\node[draw=black, fill=white, thick, circle, inner sep=1.2pt] (2) at ([shift=({-36:1in})]1) {\footnotesize $A_2$};
\node[draw=black, fill=white, thick, circle, inner sep=1.2pt] (3) at ([shift=({-108:1in})]2) {\footnotesize $A_3$};
\node[draw=black, fill=white, thick, circle, inner sep=1.2pt] (4) at ([shift=({-180:1in})]3) {\footnotesize $A_4$};
\node[draw=black, fill=white, thick, circle, inner sep=1.2pt] (5) at ([shift=({-252:1in})]4) {\footnotesize $A_5$};
\node[draw=black, fill=white, thick, circle, inner sep=1.2pt] (6) at (0,-1.2) {\footnotesize $A_6$};
\node[draw=black, fill=white, thick, circle, inner sep=1.2pt] (7) at (1.75,1) {\footnotesize $A_7$}; 
\draw[line width=2pt, black] (1)--(2)--(3)--(4)--(5)--(1) (1)--(6)--(4) (6)--(3);
\draw[black, dotted, line width=2pt] (7)--(6); 
\node at (0,-4) {\footnotesize Case 11: $G\in \mathcal{H}^*$}; 
\end{scope}
\end{tikzpicture}
\caption{Cases 9--11: $G\in \G_9^*$, $G\in G_{10}^*$, or $G\in\HH^*$.\label{case910H-fig}}
\end{figure}

%%%%%%%%%%%%%%%%%%%%%%%%%%%%%%%%%%%%%%%%%%%%%%%%%%%%%%%        H (G_11)
\textbf{Case 11: $\bm{G\in\mathcal{H}^*}$.}
Let $\HH$ be the class of connected $(P_5, \text{gem})$-free graphs
$G$ for which $V(G)$ can be partitioned into non-empty sets $A_1, \dots,
A_7$ such that the following properties all hold: each $A_i$ induces a $P_4$-free graph, the vertex-set of each
component of $G[A_7]$ is a homogeneous set, each edge with exactly one
endpoint in $V(A_7)$ has the other endpoint in $V(A_6)$, and 
for all distinct $i,j\in\{1,\dots, 6\}$, if a solid edge appears between $A_i$ and $A_j$
in the rightmost graph in Figure~\ref{case910H-fig}, 
then $[A_i,A_j]$ is complete, but if no edge
appears then $[A_i,A_j]=\emptyset$.   Let $\mathcal{H}^*$ be the class of
graphs that are in $\mathcal{H}$ with $A_1, \dots, A_5$ being cliques
and also each component of $A_7$ being a clique.  By Theorem~\ref{CKMM-thm2},
we know that if $G\in \HH$, then in fact $G\in \HH^*$.

%Since $A_6$ is $P_4$-free, $A_6$ is perfect.  Thus, $\omega(A_6)=\chi(A_6)$. 
If $\card{A_5}\leq \omega(A_6)$, then by criticality $G-A_5$ has an 8-coloring,
call it $\vph$.  To extend $\vph$ to $G$, we color $A_5$ with colors used on a
maximum clique in $A_6$. This yields a proper 8-coloring of $G$,  a
contradiction.  Thus, $\card{A_5}>\omega(A_6)$.  By symmetry,
$\card{A_2}>\omega(A_6)$. Since $A_6$ is nonempty by definition, $\card{A_2},
\card{A_5}\geq 2$.  Since $d(a_6)\leq 9$ and each of $A_1$, $A_3$, and $A_4$ is
nonempty, each component of $A_7$ has size at most 6.  %Let $G':= G- (I_1\cup I_2)$, where
If $\card{A_6}\geq 2$, then let $I_1=\{a_2,a_5,a_6\}$ and
$I_2=\{a_2',a_5',a_6'\}$.  Now $G-(I_1\cup I_2)$ is $5$-degenerate with order
$A_1', A_5', A_2', A_4', A_3', A_6', A_7'$.  By Lemma~\ref{degen-lem}, $G$ is
8-colorable, a contradiction.  So assume instead that $\card{A_6}=1$.

By criticality, there exists an $8$-coloring $\varphi$ of $G':= G- V(A_7)$.
Each component of $A_7$ is a clique of size at most $6$ that is adjacent to
only $a_6$ in $G'$. Thus we can color each component of $A_7$ with $8$ colors
so that $\varphi$ extends to $G$, which contradicts that $\chi(G)\geq
\Delta(G)=9$. 
\end{proof}

\begin{figure}[!h]
\centering
\end{figure}

%%%%%%%%%%%%%%%%%%%%%%%%%%%%%%%%%%%%%%%%%%%%%%%%%%%%
%%%%%%%%%%%%%%%                   Bibliography                   %%%%%%%%%%%%%%%%%%%
%%%%%%%%%%%%%%%%%%%%%%%%%%%%%%%%%%%%%%%%%%%%%%%%%%%%
\bibliographystyle{amsplain}

\begin{thebibliography}{10}
% Add archive and doi links in actual bibitem? 

\bibitem {BK77} O.V. Borodin and A.V. Kostochka. On an upper bound of a graph's chromatic number, depending on the graph's degree and density. \textit{Journal of Combinatorial Theory Series B}, 23, no.2-3:247-250, 1977. %https://doi.org/10.1016/0095-8956(77)90037-5

\bibitem {Brooks} R.L. Brooks. On colouring the nodes of a network. \textit{Mathematical Proceedings of the Cambridge Philosophical Society}, vol. 37, Cambridge Univ Press, pp. 194-197, 1941. 

\bibitem {Ca} P.A. Catlin. Embedding subgraphs and coloring graphs under extremal degree conditions. \textit{ProQuest LLC}, Ann Arbor, MI, Thesis (Ph.D.)-The Ohio State University, 1976. 

\bibitem {CKMM19} M.~Chudnovsky, T.~Karthick, P.~Maceli, and F. Maffray.
Coloring graphs with no induced five-vertex path or gem.
\url{https://arxiv.org/abs/1810.06186}, 2019. 
% The paper CMMK19 has NOT been published as of May 2020. There is a v1 and v2
% on archive. Make sure the newest paper is used. 

\bibitem {SPGT} M.~Chudnovsky, N.~Robertson, P.~Seymour, R.~Thomas. The strong perfect graph theorem. \textit{Annals of Mathematics Second Series}, vol. 164, no. 1, 51-229, 2006. %https://www-jstor-org.proxy.library.vcu.edu/stable/20159988

\bibitem {CR13clawfree} D.W.~Cranston and L.~Rabern. Coloring claw-free graphs with $\Delta-1$ colors. \textit{SIAM Journal of Discrete Math.}, vol. 27, no.1:534-549, 2013. %https://doi.org/10.1137/12088015X

\bibitem {CR13} D.W.~Cranston and L.~Rabern. Coloring a graph with
$\Delta-1$ colors: Conjectures equivalent to the Borodin--Kostochka Conjecture that appear weaker. \textit{European Journal of Combinatorics}, 44 Part A:23-42, 2015. %https://arxiv.org/abs/1203.5380v2

\bibitem {CKMM18} T. Karthick and F.~Maffray. Coloring (gem, co-gem)-free graphs. \textit{Journal of Graph Theory}, vol. 89, no.3:288-303, 2018. 
% https://onlinelibrary.wiley.com/doi/abs/10.1002/jgt.22251

\bibitem {King} A.D. King, Hitting all maximum cliques with a stable set using lopsided independent transversals. \textit{Journal of Graph Theory} vol. 67 no.4:300-305, 2011. %https://doi-org.proxy.library.vcu.edu/10.1002/jgt.20532

\bibitem {Ko} A.V. Kostochka. Degree, density, and chromatic number. \textit{Metody Diskret. Anal.}, no. 35:45-70, 1980, (in Russian).

\bibitem {ReedBK} B. Reed. A strengthening of Brooks' theorem. \textit{Journal of Combinatorial Theory Series B}, 76, no.2:136-149, 1999. %https://doi.org/10.1006/jctb.1998.1891

\bibitem {Reed} B. Reed. $\omega$, $\Delta$, and $\chi$. \textit{Journal of Graph Theory}, 27, no. 4:177-212, 1998. %https://doi-org.proxy.library.vcu.edu/10.1002/(SICI)1097-0118(199804)27:4<177::AID-JGT1>3.0.CO;2-K

\bibitem {West} D.B.~West. Introduction to graph theory. \textit{Prentice Hall, Inc.}, Upper Saddle River, NJ, 1996. xvi+512 pp. ISBN: 0-13-227828-6. 

\end{thebibliography}

\end{document}